\documentclass[11pt]{article}%
\usepackage{amsmath}
\usepackage{amsfonts}
\usepackage{amssymb}
\usepackage{graphicx}%
\newtheorem{theorem}[equation]{Theorem}

\newtheorem{lemma}[equation]{Lemma}

\newtheorem{proposition}[equation]{Proposition}

\newenvironment{proof}[1][Proof]{\noindent\textbf{#1.} }{\ \rule{0.5em}{0.5em}}

\def\sideremark#1{\ifvmode\leavevmode\fi\vadjust{\vbox to0pt{\vss
 \hbox to 0pt{\hskip\hsize\hskip1em
\vbox{\hsize2cm\tiny\raggedright\pretolerance10000 
 \noindent #1\hfill}\hss}\vbox to8pt{\vfil}\vss}}} 

\title{$\mathrm{Sp}(2)/\mathrm{U}(1)$ and a Positive
Curvature Problem}

\author{Ming Xu\footnote{Address: College of Mathematics,
        Tianjin Normal University,
        Tianjin 300387, P.R.China; e--mail: {\tt mgmgmgxu@163.com}.
        Research supported by NSFC no. 11271216, State Scholarship
        Fund of CSC (no. 201408120020), Science and Technology Development
        Fund for Universities and Colleges in Tianjin
        (no. 20141005), Doctor fund of Tianjin Normal
        University (no. 52XB1305).}\,\, \&
        Joseph A. Wolf\footnote{ Corresponding author.
        Address: Department of Mathematics, University of California, Berkeley,
        CA 94720--3840; e--mail: {\tt jawolf@math.berkeley.edu}.
        Research partially supported by a Simons Foundation grant and by the
        Dickson Emeriti Professorship at the University of California,
        Berkeley.}}

\date{March 9, 2015}

\begin{document}
\maketitle

\begin{abstract}
A compact Riemannian homogeneous space $G/H$, with a bi--invariant orthogonal decomposition $\mathfrak{g}=\mathfrak{h}+\mathfrak{m}$ 
is called positively curved for commuting pairs, if the sectional
curvature vanishes for any tangent plane in $T_{eH}(G/H)$ spanned by
a linearly independent commuting pair in $\mathfrak{m}$. In this paper,
we will prove that on the coset space $\mathrm{Sp}(2)/\mathrm{U}(1)$, in which $\mathrm{U}(1)$ corresponds to a short root, admits positively curved
metrics for commuting pairs. B. Wilking recently proved that 
this $\mathrm{Sp}(2)/\mathrm{U}(1)$ can not be positively curved in the general sense. This is the first example to distinguish the set of compact coset spaces admitting positively curved metrics, and that for metrics
positively curved only for commuting pairs. 
\end{abstract}

\section{Introduction}\label{sec1}
\setcounter{equation}{0}
Let $G/H$ be a compact Riemannian homogeneous space with $G$ compact. 
With respect to any bi--invariant inner product 
$\langle\cdot,\cdot\rangle_{\mathrm{bi}}$ on $\mathfrak{g}$,
there is an invariant orthogonal decomposition $\mathfrak{g}=\mathfrak{h}+
\mathfrak{m}$ of the Lie algebra of $G$, and as usual $\mathfrak{m}$ is 
identified with the tangent space $T_{eH}(G/H)$.
\smallskip

We call the Riemannian homogeneous space $G/H$ {\it positively curved for commuting pairs}, if for any linearly independent commuting pair $X$ and $Y$
in $\mathfrak{m}$, the sectional curvature of the tangent plane
$\mathrm{span}\{X,Y\}\subset T_{eH}(G/H)$ is positive. This notion 
contrasts with the traditional
algebraic method for the classification of positively curved Riemannian 
homogeneous spaces (\cite{AloffWallach1975}, \cite{BB}, \cite{Ber61}, 
\cite{Wallach1972}).
In those papers, the method for showing that a compact homogeneous space 
$G/H$ fails to have strictly positive sectional curvature, is to show that 
the sectional curvature vanishes for some commuting pair.  
It was generally accepted
that compact coset spaces admitting homogeneous metrics positively curved for 
commuting pairs are exactly the homogeneous Riemannian 
manifolds of strictly positive sectional curvature.
\smallskip

While trying to generalize these classifications to
the Finsler situation (\cite{XD}, \cite{XDHH}), we found a problem in 
L. B{\'e}rard--Bergery's classification \cite{BB} of odd 
dimensional positively curved Riemannian homogeneous spaces. There is a 
gap in the argument that the coset space 
$\mathrm{Sp}(2)/\mathrm{U}(1)$ (where $\mathrm{U}(1)$ corresponds to
a short root) cannot be positively curved.  After a stratified classification
of Cartan subalgebras contained in $\mathfrak{m}$ for this 
$\mathrm{Sp}(2)/\mathrm{U}(1)$, we saw that the traditional algebraic method
mentioned above cannot be used to exclude $\mathrm{Sp}(2)/\mathrm{U}(1)$
from the list of positively curved homogeneous spaces. Formally, we have
the following main theorem.

\begin{theorem} \label{main} Consider the compact homogeneous space
$G/H=\mathrm{Sp}(2)/\mathrm{U}(1)$ in which $H$ corresponds to a short root,
with the orthogonal decomposition $\mathfrak{g}=\mathfrak{h}+\mathfrak{m}$ for
a bi--invariant inner product. Then there are $G$--homogeneous Riemannian
metrics on it which are positively curved for commuting pairs, 
i.e. at $o=eH\in G/H$, the sectional curvature
$K(o,X\wedge Y)>0$ for any
linearly independent commuting pair $X$ and $Y$ in $\mathfrak{m}=T_o M$.
\end{theorem}

After we announced this result, B. Wilking found a way to prove 
that $\mathrm{Sp}(2)/\mathrm{U}(1)$ does not admit homogeneous Riemannian metrics of positive curvature (see Theorem \ref{thm-2} in Section \ref{sec5}).
At the same time as the problem in \cite{BB} was fixed, 
Theorem \ref{thm-2}, together with the main theorem, provides us
the first example of compact homogeneous space that is positively curved 
for commuting pairs but not positively curved in the general sense. As 
the traditional algebraic method works well in most other cases, 
non--positively curved Riemannian homogeneous spaces which are positively 
curved for commuting pairs may be very rare.  We thank Burkhard Wilking
and Wolfgang Ziller for several e--mail discussions that led us to this
refinement of our original note.

\section{The Basic Setup for $\mathrm{Sp}(2)/\mathrm{U}(1)$}\label{sec2}
\setcounter{equation}{0}
Let $M$ be the coset space $G/H=\mathrm{Sp}(2)/\mathrm{U}(1)$, in which
$H$ corresponds to a short root. We borrow the following construction from
\cite{BB} with some minor changes. Any matrix
$$\frac{1}{2}\left(
\begin{array}{cc}
u+w & v-\lambda \\
v+\lambda & u-w \\
\end{array}
\right)
$$
in $\mathfrak{g}=\mathrm{Lie}(G)=\mathfrak{sp}(2)$ can be identified with a
formal row vector $(\lambda, u,v,w)$,
in which the pure imaginary quaternions $u$, $v$, and $w$ are viewed as
column vectors in $\mathbb{R}^3$ with the more preferred
dot and cross products with respect to the standard orthonormal basis
$\{e_1,e_2,e_3\}$, instead of quaternion multiplication.
For the bi--invariant inner product of $\mathfrak{g}$, the different factors
of $\lambda$, $u$, $v$ and $w$ are orthogonal to each other, and the
restriction of the bi--invariant inner product to each factor of $u$, $v$ or
$w$ coincides with the standard inner product up to scalar changes.
The subalgebra $\mathfrak{h}=\mathrm{Lie}(H)=\mathfrak{u}(1)$ can be
identified with the subspace $u=v=w=0$, i.e. the $\lambda$--factor, and its
bi--invariant orthogonal complement $\mathfrak{m}$ can be identified with the
subspace $\lambda=0$. For any two vectors $X=(0,u,v,w)$
and $Y=(0,u',v',w')$ in $\mathfrak{m}$, their bracket can be presented as
$$[X,Y]=(v\cdot w'-v'\cdot w, u\times u'+v\times v'+ w\times w',
u\times v'-u'\times v, u\times w'-u'\times w).$$

Any $G$--homogeneous metric on $M$ can be defined from an
$\mathrm{Ad}(H)$--invariant inner product on $\mathfrak{m}$.
Our presentation of $\mathfrak{m}$ naturally splits, with
the $u$--factor corresponding to the trivial $H$--representation, and the
other two factors each corresponding to the same non--trivial irreducible
$H$--representation, i.e. for $Z=(1,0,0,0)\in\mathfrak{h}$,
$$\mathrm{Ad}(\exp(tZ))(0,0,v,w)=(0,0,\cos(2t)v+\sin(2t)w,-\sin(2t)v+\cos(2t)w).$$
So any $\mathrm{Ad}(H)$--invariant inner product
$\langle\cdot,\cdot\rangle$ on $\mathfrak{m}$ must be of the form
$\langle \cdot,\cdot\rangle=\langle \cdot,M\cdot\rangle_{\mathrm{bi}}$,
in which the linear isomorphism $M:\mathfrak{m}\rightarrow\mathfrak{m}$
satisfies,
$$M(0,u,v,w)=(0,Au,Cv-Bw,Bv+Cw),$$
where $A$ and $C$ are self adjoint, $B$ is skew adjoint, $A>0$ and
$C-\sqrt{-1}B>0$ (or equivalently $\left(
                                     \begin{smallmatrix}
                                       C & -B \\
                                       B & C \\
                                     \end{smallmatrix}
                                   \right) >0$).
To see this, we use $\mathrm{Ad}(H)$--invariance and the fact that
$\mathrm{Ad}(H)$ is trivial on the $u$--factor and rotates between the
$v$-- and $w$--factors.  So $M(0,u,v,w)$ has form
$(0,Au,B_1 v+B_2 w, B_3 v+B_4 w)$.  Since the resulting
inner product on $\mathfrak{m}$ is $\mathrm{Ad}(H)$--invariant, the
$6\times6$ matrix
$\left(
          \begin{smallmatrix}
            B_1 & B_2 \\
            B_3 & B_4 \\
          \end{smallmatrix}
\right)$
commutes with all rotations
$\left(
                                \begin{smallmatrix}
                                  \cos t I & \sin t I \\
                                  -\sin t & \cos t I \\
                                \end{smallmatrix}
\right)$.
It follows that $B_1=B_4$ and $B_2=-B_3$.  As $M$ is self adjoint and
positive definite, $M(0,u,v,w)=(0,Au,Cv-Bw,Bv+Cw)$ with $A > 0$ self adjoint,
$B$ skew adjoint, and $C$ self adjoint. Thus the action of $M$ on the
$v, w$ $6$--plane is given by
$\left(
                                     \begin{smallmatrix}
                                       C & -B \\
                                       B & C \\
                                     \end{smallmatrix}
\right)>0$.
\smallskip

In B\' erard--Bergery's argument, he missed the $B$--term. In later discussion,
we only consider small perturbations of the
$G$--normal Riemannian homogeneous metric which corresponds to
$M=M_0=\mathrm{Id}$, so we denote $M_t=I+tL$ for $t\geq 0$, in which
$L:\mathfrak{m}\rightarrow\mathfrak{m}$ is defined by
$L(0,u,v,w)=(0,Au,Cv-Bw,Bv+Cw)$ with $A$ and $C$ self adjoint, and $B$
skew adjoint. For $t$ sufficiently close to 0, the corresponding
$G$--homogeneous metric is denoted as $g_t$.

\section{Proof of the Main Theorem}\label{sec3}
\setcounter{equation}{0}
With respect to the standard basis $\{e_1,e_2,e_3\}$ of $\mathbb{R}^3$,
we have linear maps $A$, $B$ and $C$ defined by the matrices
$$A=\left(
      \begin{smallmatrix}
        0 & 1 & 0 \\
        1 & 0 & 0 \\
        0 & 0 & 0 \\
      \end{smallmatrix}
    \right),
B=\left(
  \begin{smallmatrix}
  0  & 1 & 0 \\
  -1 & 0 & 0 \\
  0  & 0 & 0 \\
  \end{smallmatrix}
\right),
\mbox{ and }
C=\left(
    \begin{smallmatrix}
      1 & 0 & 1 \\
      0 & 0 & 1 \\
      1 & 1 & 0 \\
    \end{smallmatrix}
  \right).
$$
Let $L(0,u,v,w)=(0,Au,Cv-Bw,Bv+Cw)$, $M_t = I + tL$, and
$g_t$ the corresponding $G$--invariant Riemannian metric on $M$
for $t>0$ sufficiently close to 0.
\smallskip

The sectional curvature $K^{g_t}(o,X\wedge Y)$ of $(M,g_t)$ for the tangent
plane $\mathfrak{t}=\mathrm{span}\{X,Y\}$ at $o=eH$ is
$K^{g_t}(o,X\wedge Y)=C(X,Y,t)/S(X,Y,t)$ where
$$
S(X,Y,t)=g_t(X,X)g_t(Y,Y)-g_t(X,Y)^2$$ and
\begin{eqnarray*}
C(X,Y,t)&=&-\tfrac{3}{4}\langle[X,Y]_\mathfrak{m},[X,Y]_\mathfrak{m}\rangle_{g_t}
+\tfrac{1}{2}\langle[[Y,X]_\mathfrak{m},Y]_\mathfrak{m},X\rangle_{g_t}+
\tfrac{1}{2}\langle[[X,Y]_\mathfrak{m},X]_\mathfrak{m},Y\rangle_{g_t}\\
&+&
\langle[[X,Y]_\mathfrak{h},X],Y\rangle_{g_t}
+\langle U(X,Y,t),U(X,Y,t)\rangle_{g_t}-\langle U(X,X,t),U(Y,Y,t)\rangle_{g_t\,.}
\end{eqnarray*}
Here $U:\mathfrak{m}\times\mathfrak{m}\times [0,\epsilon)\rightarrow\mathfrak{m}$
is defined by
$$\langle U(X,Y,t),Z\rangle_{g_t}=\tfrac{1}{2}(\langle[Z,X]_\mathfrak{m},Y\rangle_{g_t}+
\langle[Z,Y]_{\mathfrak{m}},X\rangle_{g_t}),$$ or equivalently (see the last
section of \cite{BB})
$$U(X,Y,t)=\tfrac{1}{2} M_t^{-1}([X,M_t Y]+[Y,M_t X]).$$
When $[X,Y]=0$ and $t=0$,
$K^{g_0}(o,X\wedge Y)=C(X,Y,0)=0$ by the sectional curvature formula for normal homogeneous spaces \cite{Ber61}, and
$\tfrac{d}{dt}C(X,Y,t)|_{t=0}=0$ because $U(X,Y,0)=0$. Thus
$\frac{d^2}{dt^2}K^{g_t}(o,X\wedge Y)|_{t=0}$ has the same sign (or 0)
as $\frac{d^2}{dt^2}C(X,Y,t)|_{t=0}$. Furthermore, when they vanish,
$\frac{d^3}{dt^3}K^{g_t}(o,X\wedge Y)|_{t=0}$ has the same sign (or 0)
as $\frac{d^3}{dt^3}C(X,Y,t)|_{t=0}$.
Direct calculation shows, when $[X,Y]=0$,
\begin{eqnarray*}
\tfrac{d^2}{dt^2}\langle U(X,Y,t),U(X,Y,t)\rangle_{g_t}|_{t=0}
=\tfrac{1}{2}
\langle[X,LY]+[Y,LX],[X,LY]+[Y,LX]\rangle_\mathrm{bi},
\end{eqnarray*}
and
\begin{eqnarray*}
\tfrac{d^2}{dt^2}\langle U(X,X,t),U(Y,Y,t)\rangle_{g_t}|_{t=0}
&=&2
\langle [X,LX],[Y,LY]\rangle_{\mathrm{bi}}=2
\langle [[X,LX],Y],LY\rangle_\mathrm{bi}\\
&=&2\langle [X,[LX,Y]],LY \rangle_{\mathrm{bi}}
=2\langle [X,LY],[Y,LX]\rangle_\mathrm{bi},
\end{eqnarray*}
thus
\begin{eqnarray}
\tfrac{d^2}{dt^2}C(X,Y,t)|_{t=0}&=&
\tfrac{1}{2}\langle[X,LY]-[Y,LX],[X,LY]-[Y,LX]
\rangle_{\mathrm{bi}}.
\end{eqnarray}
Notice that $\frac{1}{S(X,Y)^{1/2}}([X,LY]-[Y,LX])$ depends only on the tangent
plane $\mathrm{span}\{X,Y\}$.  Thus we have

\begin{lemma}\label{lemma-0}
If $X, Y \in \mathfrak{m}$ are linearly independent and commute, then
$C(X,Y,0)=\frac{d}{dt}C(X,Y,t)|_{t=0}=0$,
and
$\frac{d^2}{dt^2}C(X,Y,t)|_{t=0}\geqq 0$, with equality if and only if
$[X,LY]=[Y,LX]$. Equivalently, for any Cartan subalgebra
$\mathfrak{t}\subset\mathfrak{m}$, we have
\begin{equation}\label{1100}
K^{g_0}(o,\mathfrak{t})= \tfrac{d}{dt}K^{g_t}(o,\mathfrak{t})|_{t=0}=0,
\end{equation}
and
\begin{equation}\label{1101}
\tfrac{d^2}{dt^2}K^{g_t}(o,\mathfrak{t})|_{t=0}\geq 0
\end{equation}
with equality if an only if $[X,LY]=[Y,LX]$ in where
$\mathfrak{t}=\mathrm{span}\{X,Y\}$.
\end{lemma}

To distinguish between the situations in which
$\frac{d^2}{dt^2}C(X,Y,t)|_{t=0}$
is positive or 0, we will prove the following lemma, which is crucial for
the proof of the Theorem \ref{main}.

\begin{lemma}\label{lemma-1} Let $X, Y \in \mathfrak{m}$ linearly independent
and $\mathfrak{t} = \mathrm{span}\{X, Y\}$.  Suppose that $[X,Y] = 0$, so
$\mathfrak{t}$ is a Cartan subalgebra of $\mathfrak{g}$.
Let $\mathfrak{t}_0 = \mathrm{span}\{(0,0,e_1,0), (0,0,0,e_2)\}$.
If $\mathfrak{t} \notin \mathrm{Ad}(H)(\mathfrak{t}_0)$
then
\begin{equation}\label{1200}
\frac{d^2}{dt^2}C(X,Y,t)\Bigr |_{t=0}>0, \mbox{ or equivalently }
\frac{d^2}{dt^2}K^{g_t}(o,\mathfrak{t})\Bigr |_{t=0}>0.
\end{equation}
If $\mathfrak{t} \in \mathrm{Ad}(H)(\mathfrak{t}_0)$
then
\begin{equation*}
\frac{d^2}{dt^2}C(X,Y,t)\Bigr |_{t=0}=0 \mbox{ and } \frac{d^3}{dt^3}C(X,Y,t)\Bigr |_{t=0}>0,
\end{equation*}
or equivalently,
\begin{equation}\label{1201}
\frac{d^2}{dt^2}K^{g_t}(o,\mathfrak{t})\Bigr |_{t=0}=0 \mbox{ and } \frac{d^3}{dt^3}K^{g_t}(o,\mathfrak{t})\Bigr |_{t=0}>0
\end{equation}
\end{lemma}

The proof of Lemma \ref{lemma-1} will be postponed to the next section.
We now prove Theorem \ref{main}, assuming Lemma \ref{lemma-1}.
\smallskip

Denote the set of all Cartan subalgebras of $\mathfrak{g}$ contained in
$\mathfrak{m}$ as $\mathcal{C}$, and the set of all tangent planes at $o=eH$
as $\mathcal{G}$. Then $\mathcal{G}$ is a Grassmannian manifold,
$\mathcal{C}$ is a compact subvariety. The isotropy subgroup $H$ has natural
$\mathrm{Ad}(H)$--actions on $\mathcal{G}$ which preserve $\mathcal{C}$. It is
easy to see, for any valid $t$,
the sectional curvature function $K^{g_t}(o,\cdot)$ is $\mathrm{Ad}(H)$-invariant.
\smallskip

If $\mathfrak{t}\in\mathcal{C}$ is a Cartan subalgebra contained in $\mathfrak{m}$, such that its $\mathrm{Ad}(H)$-orbit does not contain $\mathfrak{t}_0=\mathrm{span}\{(0,0,e_1,0),
(0,0,0,e_2)\}$, then by (\ref{1200}) in Lemma \ref{lemma-1},
we can find an open neighborhood $\mathcal{U}$ of $\mathfrak{t}$ in $\mathcal{C}$,
and a positive $\epsilon$ (sufficiently close to 0, same below), such that for any Cartan subalgebra $\mathfrak{t}'\in\mathcal{U}$ and $t\in(-\epsilon,\epsilon)$,
$\frac{d^2}{dt^2}K^{g_t}(o,\mathfrak{t}')>0$. Together with (\ref{1100}) in Lemma \ref{lemma-0}, it indicates
for any Cartan subalgebra $\mathfrak{t}'\in\mathcal{U}$ and $t\in(0,\epsilon)$, $K^{g_t}(o,\mathfrak{t}')>0$.\smallskip

If $\mathfrak{t}\in\mathcal{C}$ is a Cartan subalgebra contained in $\mathfrak{m}$, such that its $\mathrm{Ad}(H)$-orbit contains $\mathfrak{t}_0$, then by (\ref{1201}), we can find an open neighborhood $\mathcal{U}$ of $\mathfrak{t}$ in $\mathcal{C}$,
and a positive $\epsilon$, such that for any Cartan subalgebra $\mathfrak{t}'\in\mathcal{U}$ and $t\in(-\epsilon,\epsilon)$,
$\frac{d^3}{dt^3}K^{g_t}(o,\mathfrak{t}')>0$.
Together with (\ref{1100}) and (\ref{1101}) in Lemma \ref{lemma-0}, it indicates
for any Cartan subalgebra $\mathfrak{t}'\in\mathcal{U}$ and $t\in(0,\epsilon)$, $K^{g_t}(o,\mathfrak{t}')>0$.\smallskip

By the compactness of $\mathcal{C}$, we can find a finite cover for it from
the open neighborhoods $\mathcal{U}$ given above, and take a uniform minimum
$\epsilon>0$. Then for any Cartan subalgebra $\mathfrak{t}\in\mathcal{C}$
contained in $\mathfrak{m}$ and $t\in(0,\epsilon)$, $K^{g_t}(o,\mathfrak{t})>0$.
This completes the proof of Theorem \ref{main}.

\section{Proof of Lemma \ref{lemma-1}}\label{sec4}
\setcounter{equation}{0}
The proof of Lemma \ref{lemma-1} is an analysis of the Cartan subalgebras
of $\mathfrak{g}$ contained in $\mathfrak{m}$.  Observe that $\mathcal{C}$
is the union of the following $\mathrm{Ad}(H)$--invariant subsets.
\begin{description}
\item{\bf Case I}. The Cartan subalgebra $\mathfrak{t}$ is spanned by
$X=(0,0,v,w)$ and $Y=(0,0,v',w')$ in $\mathfrak{m}$, it belongs to $\mathcal{C}_1$.
\item{\bf Case II}. The tangent plane $\mathfrak{t}$ is spanned by
$X=(0,u,v,w)$ and $Y=(0,0,v',w')$ in $\mathfrak{m}$, in which $u\neq 0$, it belongs to
$\mathcal{C}_2$.
\item{\bf Case III}. The tangent plane $\mathfrak{t}$ is spanned by
$X=(0,u,v,w)$ and $Y=(0,u',v',w')$ in $\mathfrak{m}$, in which $u$ and $u'$ are linearly independent, it belongs to
$\mathcal{C}_3$.
\end{description}
The two techniques we will use are change of basis in a given $\mathfrak{t}$,
and change of $\mathfrak{t}$ in $\mathcal{C}$ by the action of $H$, to
reduce our discussion to several cases with very simple $X$ and $Y$.
\smallskip

{\em Proof of Lemma \ref{lemma-1} in Case I}.
Assume that $X=(0,0,v,w)$ and $Y=(0,0,v',w')$
span the Cartan subalgebra $\mathfrak{t}$.
\smallskip

First, consider the situation where $v$ and $w$ are linearly dependent.
Changing basis of
$\mathfrak{t}$ by a suitable $\mathrm{Ad}(H)$--action, we can assume $w=0$.
Subtracting a multiple of $X$ from $Y$ we can assume $v'\cdot v=0$. Since
$[X,Y]=0$ we have $v\times v'=0$. Thus $v'=0$. Also from $[X,Y]=0$,
we have $v\cdot w'=0$. Both $v$ and $w'$ can be normalized to have length 1.
\smallskip

Next, consider the situation that $v$ and $w$ are linearly independent.
Because $[X,Y]=0$, we have
\begin{eqnarray}
v\times v'&=&-w\times w', \mbox{ and }\label{0098}\\
v\cdot w'&=&v'\cdot w. \label{0099}
\end{eqnarray}
From (\ref{0098}), $v'$ and $w'$ are contained in $\mathrm{span}\{v,w\}$.
By a suitable $\mathrm{Ad}(H)$--action, we may assume $v\cdot w=0$.
Replacing $Y$ with a suitable linear combination of $X$ and $Y$, we can
assume $v'\cdot v=0$ as well. If $v'=0$, it goes back to the last situation,
otherwise we can normalize $v$ and $v'$ and assume $|v|=|v'|=1$.
Express $w=b_2 v'$ and $w'=c_1 v+c_2 v'$, with $b_2\neq 0$.
By (\ref{0098}) and (\ref{0099}), $b_2=c_1=\pm 1$. We can further change $Y$
to $\pm Y$ and assume $b_2=c_1=1$. Then
$$X''=Y+\tfrac{1}{2}(-c_2\pm\sqrt{c_2^2+4} X)=(0,0,v'',w'')$$
where $v''$ and $w''$ are linearly independent. Replacing $X$ with
$X''$, we reduce to the last situation.
\smallskip

To summarize, for $\mathfrak{t}\subset\mathcal{C}_1$, we can find a
representative $\mathrm{span}\{(0,0,v,0),(0,0,0,w')\}$
in the $\mathrm{Ad}(H)$--orbit of
$\mathfrak{t}$, for which $|v|=|w'|=1$ and $v\cdot w'=0$.
\smallskip

Now we may suppose  $\mathfrak{t}$ is spanned by $X=(0,0,v,0)$ and
$Y=(0,0,0,w')$ with $|v|=|w'|=1$ and $v\cdot w'=0$.
If $\frac{d^2}{dt^2}C(X,Y,t)|_{t=0}=0$, i.e. $[X,LY]=[Y,LX]$, then
\begin{eqnarray}
w'\cdot Cv&=&0 \mbox{, and }\label{0101}\\
v\times Bw'&=&-w'\times Bv.\label{0100}
\end{eqnarray}
From (\ref{0100}), $B$ preserves the subspace spanned by $v$ and $w'$,
or equivalently
$\mathrm{span}\{v,w'\}^\perp$ is an eigenspace of $B$, which must be
$\mathbb{R}e_3$. So $\mathrm{span}\{v,w'\}=\mathrm{span}\{e_1,e_2\}$. Because of (\ref{0101}), and the speciality of the chosen $C$, we must have
$\{\pm v,\pm w'\}=\{\pm e_1,\pm e_2\}$, i.e., up to the action of
$\mathrm{Ad}(H)$,
$$\mathfrak{t}=\mathrm{span}\{(0,0,e_1,0),(0,0,0,e_2)\}.$$

To summarize, we have
$\frac{d^2}{dt^2}C(X,Y,t)|_{t=0}>0$
when $\mathfrak{t}\in\mathcal{C}_1$ is not contained in the
$\mathrm{Ad}(H)$--orbit of
$\mathrm{span}\{(0,0,e_1,0)$, $(0,0,0,e_2)\}$, and
$\frac{d^2}{dt^2}C(X,Y,t)|_{t=0}=0$, when $\mathfrak{t}\in\mathcal{C}_1$\,.
\smallskip

Further consider $\frac{d^3}{dt^3}C(X,Y,t)|_{t=0}$,
we only need to assume
$X=(0,0,e_1,0)$ and $Y=(0,0,0,e_2)$.
By direct calculation $[X,LY]=[X,M_t Y]=[Y,LX]=[Y,M_t X]=0$, and so
\begin{eqnarray*}
U(X,Y,t)&=& 0,\\
U(X,X,t)&=&
\bigl ( 0,\tfrac{t^2}{1-t^2}e_1+\tfrac{-t}{1-t^2}e_2,0,0 \bigr ),
\end{eqnarray*}
and $[Y, M_t Y]=(0,te_1,0,0)$.
So
\begin{eqnarray*}
C(X,Y,t)=-\langle U(X,X),U(Y,Y)\rangle_{g_t}=
        -\langle U(X,X),[Y,M_t Y]\rangle_{\mathrm{bi}}=\frac{ct^3}{1-t^2},
\end{eqnarray*}
where the constant $c>0$ comes  from the scalar relation between the standard
inner product on $\mathbb{R}^3$ and the restriction of the bi--invariant inner
product of $\mathfrak{g}$ to the $u$--factor.
Now it is obvious that
$\frac{d^3}{dt^3}C(X,Y,t)|_{t=0}>0$.
\smallskip

{\em Proof of Lemma \ref{lemma-1} in Case II}.
Assume that the Cartan subalgebra $\mathfrak{t}\in\mathcal{C}_1$ is spanned by
$X=(0,u,v,w)$ and $Y=(0,0,v',w')$ with $u\neq 0$. We normalize $u$ so that
that $|u|=1$.  Because $[X,Y]=0$, we have
$u\times v'=u\times w'=0$, i.e. $v',w'\in\mathbb{R}u$.
We can apply an element of $\mathrm{Ad}(H)$ and then scale, so that
$w'=0$ and $v'=u$. Using $[X,Y]=0$ again, we have $v\times v'=0$ and
$v'\cdot w=u\cdot w=0$. Subtract a suitable multiple of $Y$ from $X$;
we then have $v\cdot v'=0$, which implies $v=0$.
\smallskip

To summarize, the $\mathrm{Ad}(H)$--orbit of
$\mathfrak{t}\in\mathcal{C}_2$ contains a Cartan that is spanned by
$X=(0,u,0,w)$ and $Y=(0,0,u,0)$ with $|u|=1$ and $u\cdot w=0$.
\smallskip

If further we have $[X,LY]=[Y,LX]$, then direct calculation shows
\begin{eqnarray}
u\cdot Cw&=&0,\\
w\times Bu+u\times Bw&=& 0,\\
u\times (C-A)u &=& 0,\label{0102}\\
u\times Bu&=& 0.\label{0103}
\end{eqnarray}
From (\ref{0102}) and (\ref{0103}), the unit vector
$u$ is a common eigenvector of $B$, i.e. $u=\pm e_3$, and $u$ is also an eigenvector of $A-C$. But $e_3$ is not a eigenvector of $A-C$.
So in this case we always have $\frac{d^2}{dt^2}K^{g_t}(o,\mathfrak{t})|_{t=0}>0$.
\smallskip

{\em The proof of Lemma \ref{lemma-1} in Case III}.
Let $\mathfrak{t}\in\mathcal{C}_3$ be spanned by $X=(0,u,v,w)$ and
$Y=(0,u',v',w')$ with $u$ and $u'$ linearly independent.
\smallskip

We had observed that $v$, $w$, $v'$ and $w'$ are all contained in the subspace
spanned by $u$ and $u'$. By $[X,Y]=0$, we have $u\times v'=u'\times v$, from
which we see that $v$ and $v'$ are linear combinations of $u$ and $u'$.
Similarly $w$ and $w'$  are linear combinations of $u$ and $u'$.
\smallskip

Next, consider the situation where $v$ and $w$ are linearly dependent. They
cannot both vanish because the $u$--factor of $[X,Y]$ does not vanish.
Acting by a suitable $\mathrm{Ad}(H)$, we can make $w=0$. Subtracting a
suitable multiple of $X$ from $Y$, we have $v\cdot v'=0$.  Then
we can find linear combination $Y''=(0,u'',v'',w'')$ of $X$ and $Y$ to
substitute for $Y$, so that $v''$ and $w''$ are also linearly independent
and they cannot both vanish. Using a suitable generic $\mathrm{Ad}(H)$
transformation, we reduce to the situation where $\mathfrak{t}$ has basis
$X=(0,u,v,\mu_1 v)$ and $Y=(0,u',v',\mu_2 v')$ with the properties (i) $u$
and $u'$ are linearly independent, (ii) $v$ and $v'$ are nonzero vectors in
the span of $u$ and $u'$, and (using $[X,Y]=0$) $v$ and $v'$ form another basis
of $\mathrm{span}\{u,u'\}$.
\smallskip

Next we go to the general $X=(0,u,v,w)$ and $Y=(0,u',v',w')$ and reduce to
the situation above.  We may assume that $v$ and $w$ are linearly independent,
for otherwise the reduction is immediate.  Applying
$\mathrm{Ad}(H)$ we can suppose $u\cdot v=0$.
Subtracting a suitable multiple of $X$ from $Y$, we also have $u\cdot u'=0$.
With suitable scalar changes for $X$ and $Y$, we normalize $u$ and $u'$
so that $|u|=|u'|=1$.  Denote
\begin{eqnarray*}
v=b_2 u', v'=b'_1 u+b'_2 v,
w=c_1 u+c_2 u', \text{ and } w'=c'_1 u+c'_2 v.
\end{eqnarray*}
Then $[X,Y]=0$ forces
\begin{eqnarray}
b'_2&=&0,\mbox{ and}\\
b_2 c'_2 &=& b'_1 c_1. \label{1004}
\end{eqnarray}
Note that $X''=X+\lambda Y=(0,u'',v'',w'')$ has linearly dependent entries
$v''$ and $w''$ if and only if
$$\det\left(
        \begin{array}{cc}
          b'_1\lambda & b_2 \\
          c'_1\lambda+c_1 & c'_2\lambda+c_2 \\
        \end{array}
      \right)=b'_1 c'_2\lambda^2+(b'_1 c_2-c'_1 b_2)\lambda -b_2 c_1=0.$$
By (\ref{1004}), the above equation must have a real solution.
Substituting the corresponding $X''$ for $X$, we reduce the discussion to
the case $X=(0,u,v,\mu_1 v)$ and $Y=(0,u',v',\mu_2 v)$, and there
$\mathrm{span}\{u,u'\}=
\mathrm{span}\{v,v'\}$ is a two dimension subspace in $\mathbb{R}^3$.
\smallskip

If $\mu_1=\mu_2$, we can apply a suitable element of $\mathrm{Ad}(H)$ to make
them vanish. By similar tricks, we can make $u\cdot u'=0$ and $|u|=|u'|=1$.
There is a real number $\lambda$, such that
$X''=X+\lambda Y=(0,u'',v'',0)=(0,u+\lambda u',v+\lambda v',0)$ satisfies
$$
u''\cdot v''=(u'\cdot v')\lambda^2+(u\cdot v'+v\cdot u')\lambda+u\cdot v=0,
$$
because we can get $u\cdot v+u'\cdot v'=0$ from $[X,Y]=0$.
Replace $X$ with $X''$; then $u\cdot v=0$. Subtract a suitable multiple
of $X$ from $Y$; then $u\cdot u'=0$ again, i.e. $v$ is a scalar multiple
of $u'$. Also, normalize $u$ and $u'$ so that $|u|=|u'|=1$.
Express $v=\nu_1 u'$ and $v'=\nu_2 u+\nu_3 u'$. From $[X,Y]=0$, we get
$\nu_1\nu_2=1$ and $\nu_3=0$.
If we only require $u\cdot u'=0$ then by suitable scalar changes for $Y$,
we can make $\nu_1=\nu_2=1$.

In this case, we have $X=(0,u,\nu_1 u',0)$ and $Y=(0,u',\nu_2 u,0)$,
 in which $|u|=|u'|=1$, $u\cdot u'=0$ and $\nu_1\nu_2=1$.
There is another way to present $\mathfrak{t}=\mathrm{span}\{X,Y\}$
in which $\nu_1$ and $\nu_2$ do not appear.  Replace $Y$ by $\nu_1 Y$ and
$u'$ by $\nu_1 u'$. Then we have
$X=(0,u,u',0)$ and $Y=(0,u',u,0)$ in where $u\cdot u'=0$.

If further we have $[X,LY]=[Y,LX]$, then
\begin{eqnarray}
u\cdot Bu'&=&0,\label{1005}\\
u\times (A-C)u'&=&u'\times (A-C)u\\
u'\times (A-C)u'&=&u\times (A-C)u,\\
u\times Bu &=&  u'\times Bu'\label{0200}.
\end{eqnarray}
By (\ref{0200}), $B$ preserves $\mathrm{span}\{u,u'\}$, so
$\mathrm{span}\{u,u'\} = \mathrm{span}\{e_1,e_2\}$. Then $u\cdot Bu'\neq 0$,
contradicting (\ref{1005}). So in this case,
$\frac{d^2}{dt^2}C(X,Y,t)|_{t=0}>0$ when $X$ and $Y$ span $\mathfrak{t}$.
\smallskip

If $\mu_1\neq \mu_2$ then, because $[X,Y]=0$, $u\times v'=u'\times v$
and $\lambda_2 u\times v'=\lambda_1 u'\times v$, thus $u\times v'=u'\times v=0$.
Applying a suitable element of $\mathrm{Ad}(H)$ we have
$X=(0,u,\nu_1 u',0)$ and $Y=(0,u',\nu_2 u,\nu_3 u)$, with $\nu_1\nu_2=1$.
By similar argument, we may assume $\nu_1=\nu_2=1$.
\smallskip

If further $[X,LY]=[Y,LX]$ then
\begin{eqnarray}
u\times((C-A)u-\nu_3 Bu)&=& u'\times (C-A)u',\\
u\times(\nu_3(C-A)u+ Bu)&=&  u'\times Bu'.
\end{eqnarray}
It follows that $Bu$, $Bu'$, $(C-A)u$ and $(C-A)u'$ belong to
$\mathrm{span}\{u,u'\}$. So
$(\mathrm{span}\{u,u'\})^\perp$ consists of the common eigenvectors of $B$
and $C-A$.  Thus $(\mathrm{span}\{u,u'\})^\perp = \mathbb{R}e_3$.  But $e_3$
is not an eigenvector of
$C-A$. This is a contradiction. That completes the proof of Lemma
\ref{lemma-1}. \hfill $\diamondsuit$
\smallskip

As a by-product of the above argument we have the following explicit
description for Cartan subalgebras contained in $\mathfrak{m}$, for the
space $\mathrm{Sp}(2)/\mathrm{U}(1)$. It may be useful for further
study of curvature on that space.

\begin{proposition} Let $M=G/H = \mathrm{Sp}(2)/\mathrm{U}(1)$ in where
$\mathrm{U}(1)$ corresponds to a short root, and let
$\mathfrak{g}=\mathfrak{h}+\mathfrak{m}$ be the corresponding orthogonal
decomposition. Then the set $\mathcal{C}$ of all
Cartan subalgebras of $\mathfrak{g}$ contained in $\mathfrak{m}$ is the union
of four $\mathrm{Ad}(H)$-orbits with the following representatives:
\begin{description}
\item{\rm (1)} $\mathrm{span}\{X,Y\}$, with $X=(0,0,v,0)$ and
$Y=(0,0,0,w')$ such that $|v|=|w'|=1$ and $v\cdot w'=0$.
\item{\rm (2)} $\mathrm{span}\{X,Y\}$, with $X=(0,u,0,w)$ and
$Y=(0,0,u,0)$ such that $|u|=1$ and $u\cdot w=0$.
\item{\rm (3)} $\mathrm{span}\{X,Y\}$, with $X=(0,u,u',0)$ and $Y=(0,u',u,0)$ such that $u$ and $u'$ are linearly independent
    and $u\cdot u'=0$.
\item{\rm (4)} $\mathrm{span}\{X,Y\}$, with $X=(0,u,u',0)$ and $Y=(0,u',u,\mu u)$ such that $u$ and $u'$ are linearly independent and $\mu\neq 0$.
\end{description}
\end{proposition}

\section{This $\mathrm{Sp}(2)/\mathrm{U}(1)$ Cannot Be Positively Curved}
\label{sec5}
\setcounter{equation}{0}
With his permission we present the following unpublished theorem 
of B. Wilking.  This theorem came out of discussions of an early
version of this note.

\begin{theorem}\label{thm-2}
The compact homogeneous space
$G/H=\mathrm{Sp}(2)/\mathrm{U}(1)$, in which $H$ corresponds to a short root,
does not admit a homogeneous Riemannian metric with all sectional curvatures
positive.
\end{theorem}

Let $\mathfrak{g}=\mathfrak{h}+\mathfrak{m}$ be the bi--invariant orthogonal
decomposition.
Any homogeneous Riemannian metric on $G/H$ is one-to-one determined by
an $\mathrm{Ad}(H)$-invariant inner product $\langle\cdot,\cdot\rangle=\langle\cdot,M\cdot\rangle_{\mathrm{bi}}$
in which the self adjoint isomorphism $M:\mathrm{m}\rightarrow\mathrm{m}$,
with respect to $\langle\cdot,\cdot\rangle_{\mathrm{bi}}$, is
$\mathrm{Ad}(H)$-invariant and  positive definite.
\smallskip

The analytic technique in B. Wilking's proof can be summarized as the following lemma.

\begin{lemma}\label{wilking-method}
Let $G$ be a compact connected Lie group, $H$ a closed subgroup 
of $G$, and $\mathfrak{g}=\mathfrak{h}+\mathfrak{m}$ a bi-invariant orthogonal 
decomposition.  Suppose that for any $\mathrm{Ad}(H)$--equivariant
linear map $M:\mathfrak{m}\rightarrow\mathfrak{m}$, positive definite
with respect to the restriction of the bi-invariant inner product 
to $\mathfrak{m}$, there is an eigenvector $X\in\mathfrak{m}$ for the 
smallest eigenvalue of $M$, and another $Z\in\mathfrak{m}$, such that 
$\{X,Z\}$ is a linearly independent commuting pair.  Then
$G/H$ does not admit $G$-homogeneous Riemannian metrics of strictly
positive sectional curvature.
\end{lemma}

\begin{proof}
Any $G$-homogeneous Riemannian metric is determined by an inner product 
$\langle\cdot,\cdot\rangle=\langle\cdot,M\cdot\rangle_{\mathrm{bi}}$ on
$\mathfrak{m}$, where $M$ is a linear map as indicated in the statement of 
the lemma.
We will show the sectional curvature at $eH$ vanishes for the tangent plane
spanned by $X$ and $Y=M^{-1}(Z)$, where $X$ and $Z$ are
indicated by the lemma.
Denote $MX=\lambda X$. Because $\lambda>0$ is the smallest eigenvalue of $M$,
for any $X'\in\mathfrak{m}$, we have
\begin{eqnarray}
\langle X',M(X')\rangle_{\mathrm{bi}}&\geq& \lambda\langle X',X'\rangle_{\mathrm{bi}},\label{2000}\mbox{ and}\\
\langle X',M^{-1}X'\rangle_{\mathrm{bi}}&\leqq&
\lambda^{-1}\langle X',X'\rangle_{\mathrm{bi}}\label{2001}.
\end{eqnarray}
Direct calculation shows
\begin{equation*}
[X,MY]+[Y,MX]=[X,Z]-\lambda[X,Y]=-\lambda[X,Y]
\end{equation*}
is a vector in $\mathfrak{m}$ (see the last section in \cite{BB}),
i.e. $[X,Y]\in\mathfrak{m}$. It is easy to see $X$ and $Y$ are linearly
independent because $MX=\lambda X$ and $Z=MY$ are linearly independent.
\smallskip

Now apply the sectional curvature formula to the tangent plane spanned
by $X$ and $Y$, i.e.
$K(eH,X\wedge Y)=C(X,Y)/S(X,Y)$, in which $S(X,Y)>0$, and
\begin{eqnarray}
C(X,Y)&=&-\frac34\langle[X,Y]_\mathfrak{m},[X,Y]_\mathfrak{m}\rangle
+\frac12\langle[[Y,X]_\mathfrak{m},Y]_\mathfrak{m},X\rangle\nonumber\\
& &+
\frac12\langle[[X,Y]_\mathfrak{m},X]_\mathfrak{m},Y\rangle
+
\langle[[X,Y]_\mathfrak{h},X],Y\rangle\nonumber\\
& &+\langle U(X,Y),U(X,Y)\rangle-\langle U(X,X),U(Y,Y)\rangle,\label{2002}
\end{eqnarray}
where $U:\mathfrak{m}\times\mathfrak{m}\times [0,\epsilon)\rightarrow\mathfrak{m}$
is defined by
$$\langle U(X',Y'),Z'\rangle=\frac12(\langle[Z',X']_\mathfrak{m},Y'\rangle+
\langle[Z',Y']_\mathfrak{m},X'\rangle),$$ or equivalently
$$U(X',Y')=\frac12 M^{-1}([X',M Y']+[Y',M X']),$$
for any $X'$, $Y'$ and $Z'$ in $\mathfrak{m}$.
Because $X$ is an eigenvector of $M$, $U(X,X)=0$ and
$U(X,Y)=-\frac12\lambda M^{-1}([X,Y])$. Because $[X,Y]\in\mathfrak{m}$,
we can simplify (\ref{2002}) and estimate it as follows,
\begin{eqnarray}
C(X,Y)&=&-\frac{3}{4}\langle[X,Y],M([X,Y])\rangle_{\mathrm{bi}}
+\frac{1}{2}\langle[X,Y],[MX,Y]
+[X,MY]\rangle_{\mathrm{bi}}\nonumber\\
& &+\frac{1}{4}\lambda^2\langle M^{-1}([X,Y]),[X,Y]\rangle_{\mathrm{bi}}\nonumber\\
&=&-\frac{3}{4}\langle[X,Y],M([X,Y])\rangle_{\mathrm{bi}}
+\frac{1}{2}\lambda\langle[X,Y],[X,Y]\rangle_{\mathrm{bi}}
\nonumber\\
& &+\frac{1}{4}\lambda^2\langle M^{-1}([X,Y]),[X,Y]\rangle_{\mathrm{bi}}\nonumber\\
&\leqq&-\frac{3}{4}\lambda\langle[X,Y],[X,Y]\rangle_{\mathrm{bi}}
+\frac{1}{2}\lambda\langle[X,Y],[X,Y]\rangle_{\mathrm{bi}}
+\frac{1}{4}\lambda\langle [X,Y],[X,Y]\rangle_{\mathrm{bi}}\nonumber\\
&=&0,
\end{eqnarray}
in which the inequality makes use of (\ref{2000}) and (\ref{2001}).
This shows $K(eH,X\wedge Y)\leqq 0$.  That completes the proof of Lemma 
\ref{wilking-method}.
\end{proof}

Now back to $G/H=\mathrm{Sp}(2)/\mathrm{U}(1)$ in consideration, and we prove
Theorem \ref{thm-2}.
\smallskip

As mentioned earlier, the $\mathrm{Ad}(H)$-invariant linear map $M$ can be
expressed as 
$$M(0,u,v,w)=(0,Au,Cv-Bw,Bv+Cw),$$ 
where $u$, $v$ and $w$ in $\mathbb{R}^3$ are column vectors.  This 
is the standard presentation of vectors in $\mathfrak{m}$. 
Here $A$ and 
$\left( \begin{smallmatrix} C & -B \\ B & C \end{smallmatrix} \right)$
are positive definite matrices. Any eigenvalue of $M$ is either an 
eigenvalue of $A$ or an eigenvalue of 
$\left( \begin{smallmatrix} C & -B \\ B & C \end{smallmatrix} \right)$.
\smallskip

If one eigenvalue of $A$ is the smallest eigenvalue of $M$ we
can find a nonzero eigenvector $u\in\mathbb{R}^3$ accordingly for $A$.
Then $X=(0,u,0,0)$ and $Z=(0,0,u,0)$ satisfy the requirement of the lemma.
\smallskip

If one eigenvalue of 
$\left( \begin{smallmatrix} C & -B \\ B & C \end{smallmatrix} \right)$ 
is the smallest eigenvalue of $M$, let $X=(0,0,v,w)$ denote the corresponding
nonzero eigenvector of $M$. When $v$ and $w$ are linearly
dependent, we choose the nonzero vector $Z=(0,v,0,0)$ or
$Z=(0,w,0,0)$ such that $X$ and $Z$ satisfy the requirement of the lemma.
When $v$ and $w$ are linearly independent, we can find an element $h\in H$,
such that $\mathrm{Ad}(h)X=(0,0,v',w')$ such that $v'$ and $w'$ are nonzero vectors and $v'\cdot w'=0$. Take $Z=\mathrm{Ad}(h^{-1})(0,0,v',0)$, then
$X$ and $Z$ satisfy the requirement of the lemma. This completes the proof of
Theorem \ref{thm-2}.

\end{document}